\documentclass{article}
\usepackage{amsthm}
\usepackage{amssymb}
\usepackage{amsfonts}
\begin{document}

\newtheorem{theorem}{Theorem}[section]
\newtheorem{definition}{Definition}[section]
\newtheorem{corollary}[theorem]{Corollary}
\newtheorem{lemma}[theorem]{Lemma}
\newtheorem{proposition}[theorem]{Proposition}
\newtheorem{step}[theorem]{Step}
\newtheorem{example}[theorem]{Example}
\newtheorem{remark}[theorem]{Remark}

\font\sixbb=msbm6
\font\eightbb=msbm8
\font\twelvebb=msbm10 scaled 1095
\newfam\bbfam
\textfont\bbfam=\twelvebb \scriptfont\bbfam=\eightbb
                           \scriptscriptfont\bbfam=\sixbb

\newcommand{\tr}{{\rm tr \,}}
\newcommand{\linspan}{{\rm span\,}}
\newcommand{\rank}{{\rm rank\,}}
\newcommand{\diag}{{\rm Diag\,}}
\newcommand{\Image}{{\rm Im\,}}
\newcommand{\Ker}{{\rm Ker\,}}

\def\bb{\fam\bbfam\twelvebb}
\newcommand{\enp}{\begin{flushright} $\Box$ \end{flushright}}
\def\cD{{\mathcal{D}}}
\def\N{\bb N}

\title{Isometries of Grassmann spaces, II 
\thanks{This research was supported by the London Mathematical Society Grant, Research in Pairs (Scheme 4), Reference No.: 41642}
\thanks{The first author was also supported by the Engineering and Physical Sciences Research Council grant EP/M024784/1 and the Hungarian National Research, Development and Innovation Office -- NKFIH (grant no.~K115383).}
\thanks{The second author was also supported by grants N1-0061, J1-8133, and P1-0288 from ARRS, Slovenia.}}
\author{Gy\" orgy P\' al Geh\' er \footnote{Department of Mathematics and Statistics, University of Reading, Whiteknights, P.O.~Box 220, Reading RG6 6AX, United Kingdom, G.P.Geher@reading.ac.uk or gehergyuri@gmail.com}, \quad
Peter \v Semrl\footnote{Faculty of Mathematics and Physics, University of Ljubljana,
        Jadranska 19, SI-1000 Ljubljana, Slovenia; Institute of Mathematics, Physics, and Mechanics, Jadranska 19, SI-1000 Ljubljana, Slovenia, peter.semrl@fmf.uni-lj.si}
        }

\date{}
\maketitle

\begin{abstract}
Botelho, Jamison, and Moln\' ar \cite{BJM}, and Geh\' er and \v Semrl \cite{GeS} have recently described the general form of surjective isometries of Grassmann spaces of all projections of a fixed finite rank on a Hilbert space $H$. As a straightforward consequence one can characterize surjective isometries of Grassmann spaces of projections of a fixed finite corank. In this paper we solve the remaining structural problem for surjective isometries on the set $P_\infty (H)$ of all projections of infinite rank and infinite corank when $H$ is separable. The proof technique is entirely different from the previous ones and is based on the study of geodesics in the Grassmannian $P_\infty (H)$. However, the same method gives an alternative proof in the case of finite rank projections.
\end{abstract}
\maketitle

\bigskip
\noindent AMS classification: Primary: 47B49, Secondary: 54E40.

\bigskip
\noindent
Keywords: Isometry, Grassmann space, projection, subspace, gap metric, geodesic structure.


\section{Introduction and statement of the main results}

Let $H$ be a (real or complex) Hilbert space and $n$ a positive integer. We denote by $P_n (H)$ the set of all rank $n$ projections on $H$. In the case when $H$  
is an infinite-dimensional separable Hilbert space, the symbol $P_\infty (H)$ stands for the set of all projections whose images and kernels are both infinite-dimensional. By $\| \cdot \|$
we denote the usual operator norm on $B(H)$, the set of all bounded linear operators on $H$. The distance on the set of all projections induced by the
operator norm is usually called the gap metric.

In \cite{BJM}, 
Botelho,  Jamison, and  Moln\'ar described the general form of surjective isometries of $P_n(H)$ under some dimensionality constraints and under the additional assumption that $H$ is a complex Hilbert space.
Their main tool  was a non-commutative Mazur--Ulam type result on the local algebraic behaviour of surjective isometries between substructures of metric groups.
The authors of the present paper succeeded to extend this result also to the real case and all possible dimensions \cite{GeS}. 
They proved that if  $\dim H > n$ and $\dim H \neq 2n$, then  
for every surjective isometry $\phi\colon P_n (H) \to P_n (H)$ there exists a unitary or an antiunitary operator (an orthogonal operator in the real case) $U$ on $H$ such that 
$\phi (P) = UPU^\ast$,  $P \in P_n (H)$.
In the case when $\dim H = 2n$,  we have either the above form, or 
$\phi (P) = U(I-P)U^\ast$,  $P \in P_n (H)$. 

The main idea was to show that the set of certain geometric
midpoints between two projections $P$ and $Q$ is a compact manifold if and only if $P$ and $Q$ are orthogonal.
Then the structural result for orthogonality preserving maps on $P_n (H)$ was used to complete the proof. This approach
works in all cases but the case when $\dim H = 2n$. In this exceptional case orthogonality preservers may have a wild behaviour and additional tools coming from the geometry of algebraic
homogeneous spaces were needed to complete the proof. 

Using the obvious fact that for any pair of projections $P,Q$ we have
$$
\| P - Q \| = \| (I-P) - (I-Q) \|
$$
we can also solve the problem of describing the general form of surjective isometries on the set of all projections of a fixed finite corank leaving
open only the case of surjective isometries on the set of projections of infinite rank and infinite corank.
 
The aim of this paper is to solve this remaining case, thus completing the research program that started with \cite{BJM}. 

\begin{theorem}\label{glavni}
Let $H$ be an infinite-dimensional complex (real) separable Hilbert space and $\phi\colon P_\infty (H) \to P_\infty (H)$ a surjective map such that
$$
\| \phi (P) - \phi (Q) \| = \| P - Q \|
$$
for every pair $P,Q \in P_\infty (H)$. Then there exists a unitary or an antiunitary operator (orthogonal operator) $U$ on $H$ such that either we have
$$
\phi (P) = UPU^\ast
$$
for every $P \in P_\infty (H)$; or
$$
\phi (P) = U (I-P) U^\ast
$$
is satisfied for every $P \in P_\infty (H)$.
\end{theorem}

It should be emphasized here that there is an essential difference between surjective isometries of $P_n (H)$ and $P_\infty (H)$. While each surjective
isometry $\phi\colon P_n (H) \to P_n (H)$ preservers orthogonality, that is, for every pair $P,Q \in P_n (H)$ we have $PQ = 0 \iff \phi (P) \phi (Q) =0$,
this is not true for surjective isometries of $P_\infty (H)$. Indeed, let an infinite-dimensional separable Hilbert space $H$ be represented as the orthogonal direct sum of three
copies of a Hilbert space $K$, $H= K \oplus K \oplus K$, and let $\phi\colon P_\infty (H) \to P_\infty (H)$ be a bijective isometry defined by $\phi (P) = I -P$, $P\in P_\infty (H)$.
Then for $P,Q \in P_\infty (H)$ that have matrix representations
$$
P = \left[ \begin{matrix} {I & 0 & 0 \cr 0 & 0 & 0 \cr 0 & 0 & 0 \cr} \end{matrix} \right] \ \ \ {\rm and} \ \ \
Q = \left[ \begin{matrix} {0 & 0 & 0 \cr 0 & I & 0 \cr 0 & 0 & 0 \cr} \end{matrix} \right]
$$  
with respect to the above direct sum decomposition of $H$, we have $PQ=0$, while
$$
\phi (P) = \left[ \begin{matrix} {0 & 0 & 0 \cr 0 & I & 0 \cr 0 & 0 & I \cr} \end{matrix} \right] \ \ \ {\rm and} \ \ \
\phi (Q) = \left[ \begin{matrix} {I & 0 & 0 \cr 0 & 0 & 0 \cr 0 & 0 & I \cr} \end{matrix} \right]
$$ 
are obviously not orthogonal.

This difference indicates that the structural problem for isometries of $P_\infty (H)$ might be more difficult and that the methods used in the case of isometries on $P_n (H)$
do not work in the present case. However, a careful reader will notice that the proof techniques we will use here to prove the above theorem work also for isometries on $P_n (H)$.

The main motivation for studying this kind of problems comes from mathematical physics.
The Grassmann space $P_1(H)$ is used to represent the set of pure states of the quantum system, and the quantity $\tr (PQ)$ is the so-called transition probability between two pure states. 
The famous Wigner's unitary-antiunitary theorem describes the general form of transformations of $P_1(H)$ which preserve the transition probability. 
One can easily verify that $\|P-Q\| = \sqrt{1 - \tr PQ}$ holds true for every pair $P,Q \in P_1 (H)$. Hence, Wigner's theorem can be interpreted as the structural result for
isometries of $P_1 (H)$. For more detailed explanation with many other references, other motivations, and some  recent related results we refer ro \cite{BJM, Geh, GeS, Mo}.

Let $H$ be an infinite-dimensional separable Hilbert space and $n$ a positive integer. We will use the notation $\mathbb{N}_0 = \{ 0, 1,2, \ldots \}$. By $P^n (H)$ we denote the set of all projections on $H$ whose kernel is $n$-dimensional, and by $P(H)$
the set of all projections on $H$. We further set $P_0 (H) = \{ 0 \}$ and $P^0 (H) = \{ I \}$. Then clearly, $P(H)$ is a disjoint union of subsets of projections
$$
P(H) = \left( \bigcup_{n=0}^\infty \left( P_n (H) \cup P^n (H) \right) \right) \,  \cup P_\infty (H) .
$$
After having a full understanding of the structure of isometries of $P_n (H)$, $n=1,2, \ldots$, and $P_\infty (H)$ one may ask what are all surjective
isometries of $P(H)$. As we shall see this is a rather easy question once we observe that
$$
\| P - Q \| = 1
$$
whenever $P$ and $Q$ belong to two different subsets $P_n (H)$, $P^n (H)$, $n \in \mathbb{N}_0 \cup \{ \infty \}$,
while in the case that $P$ and $Q$ both belong to the same subset 
there exists a finite sequence of projections $P = P_0 , P_1 , \ldots , P_k = Q$ such that
$$
\| P_{j-1} - P_j \| < 1 , \ \ \ j= 1, \ldots , k.
$$

Using this simple observation it will be easy to prove the following.

\begin{theorem}\label{vse}
Let $H$ be an infinite-dimensional separable Hilbert space and $\phi\colon P(H) \to P(H)$ a surjective map such that
$$
\| \phi (P) - \phi (Q) \| = \| P - Q\|
$$
for all pairs $P,Q \in P(H)$. Then the set $\mathbb{N}_0$ can be written as a disjoint union $\mathbb{N}_0 = M \cup N$, and there exists a system of operators $U_0, U_1, U_2, \ldots ; U_\infty;$ $W_0, W_1, W_2, \ldots$ such that each of them is a unitary or an antiunitary operator on $H$ (orthogonal operator in the real case) and that we have
$$
\phi (P) = U_n P U_{n}^\ast \ \ \ {\rm\it and} \ \ \  \phi (Q) = W_n Q W_{n}^\ast
$$
for every $n \in M$, $P \in P_n (H)$ and $Q \in P^n (H)$;
$$
\phi (P) = U_n (I-P) U_{n}^\ast \ \ \ {\rm\it and} \ \ \  \phi (Q) = W_n (I-Q) W_{n}^\ast
$$
for every $n \in N$, $P \in P_n (H)$ and $Q \in P^n (H)$; and we have either
$$
\phi(P) = U_\infty P U_{\infty}^\ast
$$
for every $P \in P_{\infty} (H)$, or 
$$
\phi(P) = U_\infty (I-P) U_{\infty}^\ast
$$
for every $P \in P_{\infty} (H)$.
\end{theorem}

Of course, one can also ask what are all isometries of $P(H)$ in the case when $H$ is finite-dimensional, $\dim H= m$. Note that this time we do not need to assume that the isometry is surjective.
The reason is that each isometry $\phi\colon P(H) \to P(H)$
is an injective continuous map and since
$$
P(H) = \{ 0 \} \cup P_1(H) \cup \ldots \cup P_{m-1} (H) \cup \{ I \}
$$
is a finite disjoint union of Grassmann spaces and each of them is a compact and connected manifold, we see by the domain invariance theorem that $\phi$ must be automatically surjective. Moreover, comparing the dimensions of these manifolds
we see that every isometry $\phi\colon P(H) \to P(H)$ 
\begin{itemize}
\item either maps $0$ to $0$ and $I$ to $I$; or maps $0$ to $I$ and $I$ to $0$,
\item either maps $P_1(H)$ onto $P_1(H)$ and $P_{m-1}(H)$ onto $P_{m-1} (H)$; or maps $P_1(H)$ onto $P_{m-1}(H)$ and $P_{m-1}(H)$ onto $P_1 (H)$,
\item either maps $P_2(H)$ onto $P_2(H)$ and $P_{m-2}(H)$ onto $P_{m-2} (H)$; or maps $P_2(H)$ onto $P_{m-2}(H)$ and $P_{m-2}(H)$ onto $P_2 (H)$,
\end{itemize}
$$
\vdots
$$
We can now easily get the full understanding of the structure of isometries on $P(H)$ using the main theorem from \cite{GeS}.
In particular, one needs to know what is the general form of all isometries from $P_k (H)$ into $P_{m-k}(H)$ when $k < m/2$. All one needs to observe is that
if $\varphi\colon P_k (H) \to P_{m-k}(H)$ is an isometry then the map $\xi\colon P_k (H) \to P_k (H)$ defined by
$$
\xi (P) = I - \varphi (P), \ \ \ P \in P_k (H),
$$
is an isometry, too, and hence, the main theorem from \cite{GeS} implies that 
there exists a unitary or an antiunitary operator (orthogonal operator) $U$ on $H$ such that 
$\varphi (P) = U(I-P)U^\ast$,  $P \in P_k (H)$.

To avoid using too much space for trivialities we leave the exact formulation of the structural results for isometries of $P(H)$, $\dim H = m$, (there is a slight difference between the cases when $m$ is even or $m$ is odd)
to the reader. In the rest of the paper we will often speak of unitary operators, unitary or antiunitary operators, and unitary similarities without mentioning the real case -- it will be self-understood that this will mean orthogonal operators and orthogonal similarities.

\section{Preliminary results}

One of the main tools we will need all the time is the famous Halmos' two projections theorem, see \cite{BS}, which reads as follows.

\begin{theorem}[P.R.~Halmos]\label{twop}
	Let $H$ be a real or complex Hilbert space and
	$P,Q$ be projections on $H$. Then up to a unitary similarity $H$ can be written as an orthogonal direct sum decomposition $H= H_1 \oplus H_2 \oplus H_3 \oplus H_4 \oplus K \oplus K$
	(note that the last two summands are equal)  and the projections $P$ and $Q$ have the corresponding matrix representations
	\begin{equation}\label{twop-form}
	P = \left[ \begin{matrix} { I & 0 & 0 & 0 & 0 & 0 \cr  0 & 0 & 0 & 0 & 0 & 0 \cr  0 & 0 & I & 0 & 0 & 0 \cr  0 & 0 & 0 & 0 & 0 & 0 \cr  0 & 0 & 0 & 0 & I & 0 \cr  0 & 0 & 0 & 0 & 0 & 0 \cr} \end{matrix} \right]
	\ \ \ {\rm\it and} \ \ \ 
	Q = \left[ \begin{matrix} { 0 & 0 & 0 & 0 & 0 & 0 \cr  0 & I & 0 & 0 & 0 & 0 \cr  0 & 0 & I & 0 & 0 & 0 \cr  0 & 0 & 0 & 0 & 0 & 0 \cr  0 & 0 & 0 & 0 & C^2 & SC \cr  0 & 0 & 0 & 0 & SC & S^2 \cr} \end{matrix} \right],
	\end{equation}
	where $S, C\colon K \to K$ are self-adjoint injective operators satisfying $0 \le S,C \le I$ and $S^2 + C^2 = I$.
	Moreover, we have $H_1 = {\rm Im}\, P\cap {\rm Ker}\, Q$, $H_2 = {\rm Ker}\, P\cap {\rm Im}\, Q$, $H_3 = {\rm Im}\, P\cap {\rm Im}\, Q$ and $H_4 = {\rm Ker}\, P\cap {\rm Ker}\, Q$.
\end{theorem}

Note that it may happen that in the above representation of $P$ and $Q$ some of the subspaces $H_1, H_2 , H_3 , H_4, K$ are zero subspaces. Observe also that $S$ is the unique positive square root of $I - C^2$.
In particular, $S$ and $C$ commute. 

Next, using Halmos' two projections theorem, we prove an inequality that is probably known, however, we were not able to find it in the literature.

\begin{proposition}\label{ninana-es-jiba}
Let $H$ be a real or complex Hilbert space and $P,Q \in P(H)$.
If ${\rm Ker}\, P\cap {\rm Im}\, Q \neq \{0\}$ or ${\rm Im}\, P\cap {\rm Ker}\, Q \neq \{0\}$, then 
$$
	\| P - Q \| = 1,
$$
otherwise we have
$$
	\| P - Q \| = \sqrt{ \| (I-P)Q (I-P) \| } = \|S\|,
$$
where $S$ is the operator from (\ref{twop-form}) in the previous theorem.
In particular, we always have the following inequality:
$$
	\| P - Q \| \ge \sqrt{ \| (I-P)Q (I-P) \| }.
$$
\end{proposition}

\begin{proof}
	Without loss of generality we may assume that $P$ and $Q$ have the matrix representations (\ref{twop-form}), in which case we have
	$$
		P - Q = \left[ \begin{matrix} { I & 0 & 0 & 0 & 0 & 0 \cr  0 & -I & 0 & 0 & 0 & 0 \cr  0 & 0 & 0 & 0 & 0 & 0 \cr  0 & 0 & 0 & 0 & 0 & 0 \cr  0 & 0 & 0 & 0 & S^2 & -SC \cr  0 & 0 & 0 & 0 & -SC & -S^2 \cr} \end{matrix} \right].
	$$
	Thus, if $H_1 \neq \{0\}$ or $H_2 \neq \{0\}$, then clearly $\|P-Q\| = 1$.
	On the other hand, if $H_1 = H_2 = \{0\}$, then we have
	$$
		\|P-Q\| = \left\|\left[ \begin{matrix} { S^2 & -SC \cr -SC & -S^2 \cr} \end{matrix} \right]\right\| = \left\|\left[ \begin{matrix} { S & -C \cr -C & -S \cr} \end{matrix} \right] \left[ \begin{matrix} { S & 0 \cr 0 & S \cr} \end{matrix} \right]\right\| = \left\|\left[ \begin{matrix} { S & 0 \cr 0 & S \cr} \end{matrix} \right]\right\| $$ $$= \|S\| = \sqrt{ \| S^2 \| } = \sqrt{\left\|\left[ \begin{matrix} { 0 & 0 & 0 & 0 \cr 0 & 0 & 0 & 0 \cr 0 & 0 & 0 & 0 \cr 0 & 0 & 0 & S^2 \cr} \end{matrix} \right]\right\|}
		 = \sqrt{ \| (I-P)Q (I-P) \| },
	$$
	where we observed that $\left[ \begin{matrix} { S & -C \cr -C & -S \cr} \end{matrix} \right]$ is a unitary operator.
	As we obviously have $\| (I-P)Q (I-P) \| \leq 1$, the stated inequality is straightforward.
\end{proof}

From now on till the end of this section $H$ will always denote an infinite-dimensional real or complex separable Hilbert space.

\begin{lemma}\label{edmon}
Let $x \in H$ be a vector of norm one and assume that matrix representations of $P,Q \in P_\infty (H)$ with respect to the orthogonal
direct sum decomposition $H = {\rm span}\, \{x \} \oplus x^\perp$ are
$$
P = \left[ \begin{matrix} { 1 & 0 \cr 0 & P_1 \cr} \end{matrix} \right] \ \ \ {\rm\it and} \ \ \ 
Q = \left[ \begin{matrix} { 0 & 0 \cr 0 & Q_1 \cr} \end{matrix} \right]
$$
for some projections $P_1 , Q_1 \in P_\infty (x^\perp )$. Let $\theta \in (0, {\pi \over 2} )$. Assume further that $R \in P_\infty (H)$
satisfies
$$
\| R - P \| \le \sin \theta \ \ \ {\rm\it and} \ \ \ \|R - Q \| \le \cos \theta .
$$
Then there exists a vector $y\in H$ of norm one such that $x \perp y$ and
 the matrix representations of $P,Q,R$ with respect to the orthogonal
direct sum decomposition $H = {\rm span}\, \{x \} \oplus {\rm span}\, \{ y \} \oplus \{ x, y \}^\perp$ are
$$
P = \left[ \begin{matrix} { 1 & 0 & 0 \cr 0& 0 & 0 \cr 0 & 0 & P_2 \cr} \end{matrix} \right] , \ \ \ 
Q = \left[ \begin{matrix} { 0 & 0 & 0 \cr 0 & 1 & 0  \cr0 &  0 & Q_2 \cr} \end{matrix} \right] ,
$$
and
$$
R =  \left[ \begin{matrix} { \cos^2 \theta & \cos \theta \sin \theta & 0 \cr  \cos \theta \sin \theta & \sin^2 \theta & 0 \cr 0 & 0 & R_2 \cr} \end{matrix} \right]
$$
for some $P_2 , Q_2 , R_2 \in P_\infty ( \{ x,y \}^\perp )$.
Moreover as a consequence, we have that $\| R - P \| = \sin \theta$ and $\|R - Q \| = \cos \theta$.
\end{lemma}

\begin{proof}
Let the matrix representation of $R$  with respect to the orthogonal
direct sum decomposition $H = {\rm span}\, \{x \} \oplus x^\perp$ be
$$
R = \left[ \begin{matrix} {r_1 & u^\ast \cr u & R_1 \cr} \end{matrix} \right].
$$
From $R^2 = R$ we get
\begin{equation}\label{kojamuz}
r_{1}^2 + \| u \|^2 = r_1 .
\end{equation}
Further, the inequality
$$
\| R - P \| = \left\|  \left[ \begin{matrix} {r_1 -1 & u^\ast \cr u & R_1 - P_1 \cr} \end{matrix} \right]
\right\| \le \sin \theta
$$
yields
\begin{equation}\label{notsogood}
(r_1 - 1)^2 + \| u \|^2 \le \sin^2 \theta
\end{equation}
and similarly,
\begin{equation}\label{dammgood}
r_{1}^2 + \| u \|^2 \le \cos^2 \theta .
\end{equation}
It follows from (\ref{kojamuz}) and (\ref{dammgood}) that $r_1  \le \cos^2 \theta$, while 
(\ref{kojamuz}) and (\ref{notsogood}) imply $1- r_1 \le \sin^2 \theta$, or equivalently $r_1 \ge \cos^2 \theta$.

Thus, $r_{1} = \cos^2 \theta$, and consequently, by (\ref{kojamuz}) we have $\| u \|^2 = \cos^2 \theta \sin^2 \theta$.
Setting $y  = (\cos \theta \sin \theta )^{-1} u$, the matrix representations
of $P,Q,R $ with respect to the orthogonal decomposition $H = \linspan\{x\} \oplus \linspan\{y\} \oplus \{ x , y \}^\perp$ are
$$
P = \left[ \begin{matrix} { 1 & 0 & 0 \cr 0 & p_2 & v^\ast \cr 0 & v & P_2 \cr} \end{matrix} \right] , \ \ \ 
Q = \left[ \begin{matrix} { 0 & 0 & 0 \cr 0 & q_2 & w^\ast \cr 0 & w & Q_2 \cr} \end{matrix} \right] , $$ and $$
R= \left[ \begin{matrix} { \cos^2 \theta & \cos \theta \sin \theta & 0 \cr \cos \theta \sin \theta & r_2 & z^\ast \cr 0 & z & R_2 \cr} \end{matrix} \right] ,
$$
for some $p_2, q_2, r_2 \in [0,1]$, $v,w,z \in \{ x,y \}^\perp$, and some $P_2, Q_2, R_2 \in B(\{ x , y \}^\perp)$.
It follows from $R^2 = R$ that
$$
\cos^3 \theta \sin \theta + r_2 \cos \theta \sin \theta = \cos \theta \sin \theta .
$$
Therefore, $r_2 = \sin^2 \theta$, and then applying $R^2= R$ once more, we conclude that $z=0$.

The inequality
$$
\| R - Q \| \le \cos \theta
$$
implies that
$$
\left\| 
 \left[ \begin{matrix} {  \cos^2 \theta & \cos \theta \sin \theta  \cr \cos \theta \sin \theta & \sin^2 \theta - q_2 \cr } \end{matrix} \right]
\right\| \le \cos\theta .
$$
Denoting $$A =  \left[ \begin{matrix} {  \cos^2 \theta & \cos \theta \sin \theta  \cr \cos \theta \sin \theta & \sin^2 \theta - q_2 \cr } \end{matrix} \right]$$
we have ${\rm tr}\, A = 1 - q_2 \ge 0$ and $\det A = - q_2 \cos^2 \theta$. Because $\| A \| < 1$ we have $q_2 \not=0$, and consequently,
$\det A < 0$ which further yields that $A$ has one positive and one negative eigenvalue. Since the trace of $A$ is nonnegative, the norm of $A$ equals the positive
eigenvalue. The characteristic polynomial of $A$ is of the form $p(X) = X^2 - ({\rm tr}\, A) X + \det A$, and hence
$$
2 \| A \| = {\rm tr}\, A + \sqrt{ ({\rm tr}\, A)^2 - 4 \det A} = 1 - q_2 + \sqrt{ (1-q_2 )^2 + 4 q_2 \cos^2 \theta } .
$$
It is straightforward to check that the derivative of the real function $f$ defined on the unit interval $[0,1]$ by
$$
f(t) =   1 - t + \sqrt{ (1-t )^2 + 4 t \cos^2 \theta }
$$
is negative, and thus the minimal value of the function $f$ on the unit interval is $f(1) = 2 \cos \theta$. Since $2 \| A \| \le 2 \cos \theta$ we conclude that $q_2 =1$.
Then $Q \le I$ immediately yields that $w=0$.

In the same way we verify that $p_2 = 0$ and $v=0$.
\end{proof}

For any projection $P \in P_\infty (H)$ and any positive real number $c$ we denote 
$$
P^{\le c} = \{ R \in P_\infty (H) \colon \, \| R - P \| \le c \}.
$$

Let $P,Q \in P_\infty (H)$. Set $H_1 = {\rm Im}\, P \cap {\rm Ker}\, Q$ and $H_2 = {\rm Ker}\, P \cap {\rm Im}\, Q$. Note that one or both of these two subspaces might be trivial. Obviously, they are orthogonal and both invariant for both $P$ and $Q$. Hence, with respect to the orthogonal direct sum decomposition
$ H = H_1 \oplus H_2 \oplus H_3$ the projections $P$ and $Q$ have the matrix representations
$$
P = \left[ \begin{matrix} { I & 0 & 0 \cr 0 & 0 &0 \cr 0 & 0 & P_1} \end{matrix} \right] \ \ \ {\rm and} \ \ \
Q = \left[ \begin{matrix} { 0 & 0 & 0 \cr 0 & I &0 \cr 0 & 0 & Q_1} \end{matrix} \right],
$$
where $P_1$ and $Q_1$ are projections on $H_3$ with the property that  ${\rm Im}\, P_1 \cap {\rm Ker}\, Q_1 = \{ 0 \}$ and ${\rm Ker}\, P_1 \cap {\rm Im}\, Q_1 = \{ 0 \}$.

\begin{lemma}\label{etko}
Let $P$ and $Q$ be as above and $\theta \in (0, {\pi \over 2} )$. Let $R\in P_\infty (H)$ and assume that
$$
R \in P^{\le \sin \theta} \cap Q^{\le \cos \theta}.
$$
Then there exists a unitary operator $U\colon H_1 \to H_2$ and a projection $R_1$ on $H_3$ such that
$$
R = \left[ \begin{matrix} { (\cos^2 \theta) I & (\sin \theta \cos \theta) U^\ast & 0 \cr (\sin \theta \cos \theta) U & (\sin^2 \theta)  I &0 \cr 0 & 0 & R_1} \end{matrix} \right] .
$$
In particular, if $\dim H_1 \not= \dim  H_2$, then the set $P^{\le \sin \theta} \cap Q^{\le \cos \theta}$ is empty.
\end{lemma}

\begin{proof}
Take any unit vector $x \in H_1$. By Lemma \ref{edmon} we have $Rx = (\cos^2 \theta) x + v$, where $v$ is some vector in $H_2$ with norm $\cos \theta \sin \theta$. Hence,
the restriction of $R - (\cos^2 \theta) I$ to $H_1$ maps $H_1$ into $H_2$. This restriction is a linear isometry multiplied by $\cos \theta \sin\theta$. It follows that $R$ must be of the form
$$
R = \left[ \begin{matrix} { (\cos^2 \theta) I & * & * \cr (\sin \theta \cos \theta) U & * &* \cr 0 & * & * \cr} \end{matrix} \right] ,
$$ 
where $U$ is an isometric embedding of $H_1$ into $H_2$.
In exactly the same  way we see that
$$
R = \left[ \begin{matrix} { * & (\sin \theta \cos \theta) V & * \cr * & (\sin^2 \theta)  I & *\cr * & 0 & *} \end{matrix} \right] 
$$
with $V$ being an isometric embedding of $H_2$ into $H_1$. It follows that $H_1$ and $H_2$ are of the same dimension. Since $R$ is self-adjoint, we have necessarily $V = U^\ast$. In
particular, $U$ is surjective. It follows that $R$ is of the desired form.
\end{proof}

\begin{lemma}\label{and}
Let $P,Q,R \in P_\infty (H)$ with $P$ and $Q$ orthogonal and 
$$
\| R - P \| = \| R - Q \| = {1 \over \sqrt{2}}.
$$
Then there exists exactly one mapping $\gamma\colon [ 0, {\pi \over 2} ] \to P_\infty (H)$ such that
$$
\gamma (0) = P, \ \ \ \gamma \left({\pi \over 2} \right) = Q, \ \ \ {\rm\it and} \ \ \  \gamma \left({\pi \over 4} \right) = R,
$$
and
\begin{equation}\label{usu}
\| \gamma (\theta_1 ) - \gamma (\theta_2) \| = \sin | \theta_1 - \theta_2 |
\end{equation}
for all $\theta_1 , \theta_2 \in  [ 0, {\pi \over 2} ]$.

Moreover, the exact same statement holds if $I-P$ and $I-Q$ are orthogonal.
\end{lemma}

\begin{proof}
Since $P$ and $Q$ are orthogonal we have an orthogonal direct sum decomposition $H = H_1 \oplus H_2 \oplus H_3$ (with $H_3$ possibly a zero subspace) such that the corresponding matrix 
representations are
$$
P = \left[ \begin{matrix} { I & 0 & 0 \cr 0 & 0 & 0 \cr 0 & 0 & 0 \cr} \end{matrix} \right] \ \ \ {\rm and} \ \ \ 
Q = \left[ \begin{matrix} { 0 & 0 & 0 \cr 0 & I & 0 \cr 0 & 0 & 0 \cr} \end{matrix} \right] .
$$
We need to prove the existence and the uniqueness of $\gamma$ with the above properties. By (\ref{usu}), we will have
\begin{equation}\label{muci}
\gamma( \theta) 
\in P^{\le \sin \theta} \cap Q^{\le \cos \theta}
\end{equation}
for all $\theta \in  [ 0, {\pi \over 2} ]$
and thus, by Lemma \ref{etko}, we will further have
$$
\gamma (\theta) = \left[ \begin{matrix} { * & * & 0 \cr * & * &0 \cr 0 & 0 & T} \end{matrix} \right] 
$$
for some projection $T$ on $H_3$. But $T$ has to be zero. Indeed, if $T$ was nonzero we would have $\| \gamma(\theta) - P \| = 1 = \| \gamma(\theta) - Q \|$.
Hence, all projections $\gamma (\theta)$ will have nonzero entries only in the upper-left $2 \times 2$ corner. In other words, there is no loss of generality in assuming
that $H_3 = {0}$, and then
$$
P = \left[ \begin{matrix} { I  & 0 \cr 0 & 0  \cr} \end{matrix} \right] \ \ \ {\rm and} \ \ \ 
Q = \left[ \begin{matrix} { 0 & 0  \cr 0 & I  \cr} \end{matrix} \right] .
$$
Also, since $H_1$ and $H_2$ are of the same dimension, they can be identified, and thus, $H$ can be considered as the orthogonal direct sum of two copies of $H_1$.

Using Lemma \ref{etko} once more we see that $R$ must be of the form
$$
R = \left[ \begin{matrix} { {1 \over 2} I  &  {1 \over 2} U \cr  {1 \over 2} U^\ast &    {1 \over 2} I\cr} \end{matrix} \right ]
$$
for some unitary operator $U$ on $H_1$. Replacing $P$, $Q$, and $R$ by $WPW^\ast$, $WQW^\ast$, and $WRW^\ast$, where $W$ is the unitary operator given by
$$
W =  \left[ \begin{matrix} { I  &  0 \cr  0 &    U \cr} \end{matrix} \right ],
$$
we may assume that
$$
R = \left[ \begin{matrix} { {1 \over 2} I  &  {1 \over 2} I \cr  {1 \over 2} I &    {1 \over 2} I\cr} \end{matrix} \right ].
$$

It is an elementary linear algebra exercise to show that the map $\gamma\colon [ 0, {\pi \over 2} ] \to P_\infty (H)$ defined by
$$
\gamma (\theta ) = \left[ \begin{matrix} { (\cos^2 \theta) I  & (\cos \theta \sin \theta ) I \cr
 (\cos \theta \sin \theta ) I & (\sin^2 \theta ) I \cr} \end{matrix} \right], \ \ \ \theta \in 
\left[ 0, {\pi \over 2} \right],
$$
satisfies (\ref{usu}). Thus, the proof will be completed if we show that for any $\theta \in ( 0, {\pi \over 2})$ the projection $\gamma (\theta )$ above is the unique
projection satisfying (\ref{usu}) with $\theta_1 = \theta$ and $\theta_2 = 0, {\pi \over 4}, {\pi \over 2}$.

So, choose $\theta \in ( 0, {\pi \over 2})$ and assume that $S\in P_\infty (H)$
is a projection such that 
$$
S \in 
 P^{\le \sin \theta} \cap Q^{\le \cos \theta}
$$
and $\| S - R \| = \sin |{\pi \over 4} - \theta |$. Then, by  Lemma \ref{etko} we have
$$
S =  \left[ \begin{matrix} { (\cos^2 \theta) I  & (\cos \theta \sin \theta ) V^\ast \cr
 (\cos \theta \sin \theta ) V & (\sin^2 \theta ) I \cr} \end{matrix} \right] 
$$
for some unitary operator $V$. We will complete the proof by showing that 
$\| S - R \| = \sin |{\pi \over 4} - \theta |$ implies that $V=I$.

Observe that
$$
\left[ \begin{matrix} { {1 \over \sqrt{2}} I &  {1 \over \sqrt{2}}I \cr -  {1 \over \sqrt{2}} I &  {1 \over \sqrt{2}} I \cr} \end{matrix} \right] \, 
 R  \,
\left[ \begin{matrix} { {1 \over \sqrt{2}} I &  - {1 \over \sqrt{2}}I \cr   {1 \over \sqrt{2}} I &  {1 \over \sqrt{2}} I \cr} \end{matrix} \right] 
=  \left[ \begin{matrix} { I  &  0 \cr  0 &    0 \cr} \end{matrix} \right ],
$$
and
$$
\left[ \begin{matrix} { {1 \over \sqrt{2}} I &  {1 \over \sqrt{2}}I \cr -  {1 \over \sqrt{2}} I &  {1 \over \sqrt{2}} I \cr} \end{matrix} \right] \, 
 S  \,
\left[ \begin{matrix} { {1 \over \sqrt{2}} I &  - {1 \over \sqrt{2}}I \cr   {1 \over \sqrt{2}} I &  {1 \over \sqrt{2}} I \cr} \end{matrix} \right] 
=  \left[ \begin{matrix} { * &  * \cr  * &    {1 \over 2} [ I - \cos \theta \sin \theta (V + V^\ast) ] \cr} \end{matrix} \right ],
$$
and therefore, by Proposition \ref{ninana-es-jiba} we have
$$
 \sin \left|{\pi \over 4} - \theta \right| = \| S - R \| = \left \|   \left[ \begin{matrix} { I  &  0 \cr  0 &    0 \cr} \end{matrix} \right ] - 
 \left[ \begin{matrix} { * &  * \cr  * &    {1 \over 2} [ I - \cos \theta \sin \theta (V + V^\ast) ] \cr} \end{matrix} \right ] \right\|
$$
$$
\ge \sqrt{ \left\|  {1 \over 2} I - {1 \over 2} \cos \theta \sin \theta (V + V^\ast) \right\| } \ge \sqrt{ {1 \over 2} - \cos \theta \sin \theta } = 
\sin \left|{\pi \over 4} - \theta \right|.
$$
Since all inequalities are actually equalities, we have $V=I$, as desired.

For the case when $I-P$ and $I-Q$ are orthogonal we only have to observe that $P\mapsto I-P$ is a bijective isometry of $P_\infty(H)$.
\end{proof}

\begin{lemma}\label{slomk}
Let $L$ be a Hilbert space (finite or infinite-dimensional)
having the orthogonal direct sum decomposition $L = K \oplus K$ and assume that $P,Q \in P_{\dim K} (L)$ have the corresponding matrix
representations
$$
P = \left[ \begin{matrix} { I & 0 \cr 0 & 0 \cr} \end{matrix} \right] \ \ \ {\rm\it and} \ \ \ 
Q =  \left[ \begin{matrix} { C^2 & SC \cr SC & S^2 \cr} \end{matrix} \right]
$$
with $S, C$ self-adjoint injective operators satisfying $0 \le S,C \le I$ and $S^2 + C^2 = I$. Denote $\| S \| = \sin \psi$, $0 \le \psi \le { \pi \over 2}$.
We define a map $\alpha\colon [0 , {\pi \over 2}] \to B(L)$ by
$$
\alpha (\theta ) 
=\left[ \begin{matrix} {
\cos^2 \left( {2 \theta \over \pi} \arcsin S \right) & \cos  \left( {2 \theta \over \pi} \arcsin S \right)  \sin  \left( {2 \theta \over \pi} \arcsin S \right)  \cr
\cos  \left( {2 \theta \over \pi} \arcsin S \right)  \sin  \left( {2 \theta \over \pi} \arcsin S \right) & \sin^2  \left( {2 \theta \over \pi} \arcsin S \right) \cr
} \end{matrix} \right] 
$$
for $\theta \in \left[ 0, {\pi \over 2} \right]$.
Then 
$$
\alpha (\theta) \in P_{\dim K} (L)
$$
for every $\theta \in [ 0 , {\pi \over 2}]$,
$$
\alpha (0) = P , \ \ \ \alpha \left( {\pi \over 2} \right) = Q,
$$
and
$$
\| \alpha ( \theta_1 ) - \alpha (\theta_2 ) \| = \sin \left( {2 |\theta_1 - \theta_2 | \over \pi} \psi  \right) 
$$
for all pairs $\theta_1 , \theta_2 \in [0, {\pi \over 2}]$.
\end{lemma}

\begin{proof}
The verification of 
$\alpha (0) = P$ and $\alpha \left( {\pi \over 2} \right) = Q$ is straightforward, since $\arcsin S = \arccos C$.
It is also clear that $\alpha (\theta)$ is a projection for all
$\theta \in [ 0 , {\pi \over 2}]$. If $\Gamma \subset P(H)$ is a curve in the set of all projections then
all projections on this curve have the same rank and the same corank. Consequently, $\alpha (\theta) \in P_{\dim K} (L)$ for all
$\theta \in [ 0 , {\pi \over 2}]$.

We set 
$$
C_\theta = \cos \left( {2 \theta \over \pi} \arcsin S \right) , \ \ \   S_\theta = \sin \left( {2 \theta \over \pi} \arcsin S \right),
$$
and
$$
U_\theta = \left[ \begin{matrix} { C_\theta & - S_\theta \cr S_\theta & C_\theta \cr} \end{matrix} \right] .
$$
Clearly, each $U_\theta$ is a unitary operator. Observe that
$$
\alpha ( \theta ) = U_\theta \left[ \begin{matrix} { I & 0 \cr 0 & 0 \cr} \end{matrix} \right] U_{\theta}^\ast =
 U_\theta \alpha (0) U_{\theta}^\ast .
$$
Moreover, for every pair $\theta_1 , \theta_2$ satisfying $ 0 \le \theta_1 < \theta_2 \le {\pi \over 2}$ we have
$$
U_{\theta_1}^\ast U_{\theta_2} =  
 \left[ \begin{matrix} { C_{\theta_1} C_{\theta_2} +   S_{\theta_1} S_{\theta_2} &  S_{\theta_1} C_{\theta_2} -   C_{\theta_1} S_{\theta_2} 
 \cr  
 C_{\theta_1} S_{\theta_2} -   S_{\theta_1} C_{\theta_2}
&     C_{\theta_1} C_{\theta_2} +   S_{\theta_1} S_{\theta_2}
\cr} \end{matrix} \right] 
$$
$$
=  \left[ \begin{matrix} { C_{\theta_2 - \theta_1} & - S_{\theta_2 - \theta_1} \cr S_{\theta_2 - \theta_1} & C_{\theta_2 - \theta_1} \cr} \end{matrix} \right] = U_{\theta_2 - \theta_1}.
$$
Therefore,
$$
\| \alpha (\theta_1) - \alpha (\theta_2) \| = \left\| U_{\theta_1}  \left[ \begin{matrix} { I & 0 \cr 0 & 0 \cr} \end{matrix} \right] U_{\theta_1}^\ast
-  U_{\theta_2}  \left[ \begin{matrix} { I & 0 \cr 0 & 0 \cr} \end{matrix} \right] U_{\theta_2}^\ast \right\|
$$
$$
=\left\|  \left[ \begin{matrix} { I & 0 \cr 0 & 0 \cr} \end{matrix} \right] - U_{\theta_1}^\ast 
U_{\theta_2}  \left[ \begin{matrix} { I & 0 \cr 0 & 0 \cr} \end{matrix} \right] (U_{\theta_1}^\ast U_{\theta_2})^\ast \right\|
$$
$$
= \left\|  \left[ \begin{matrix} { I & 0 \cr 0 & 0 \cr} \end{matrix} \right] - 
U_{\theta_2 - \theta_1}  \left[ \begin{matrix} { I & 0 \cr 0 & 0 \cr} \end{matrix} \right]  U_{\theta_2 - \theta_1}^\ast \right\|
= \| \alpha (0) - \alpha ( \theta_2 - \theta_1 ) \| .
$$
Consequently, by Proposition \ref{ninana-es-jiba} we have
$$
\| \alpha (\theta_1) - \alpha (\theta_2) \|
= \left\| \sin \left[ { 2 ( \theta_2 - \theta_1) \over \pi} \arcsin S \right] \right\| = \sin \left( { 2 ( \theta_2 - \theta_1) \over \pi} \psi \right),
$$
as desired.
\end{proof}

\begin{lemma}\label{paseen}
Assume that $P,Q \in P_\infty (H)$ have matrix representations
$$
P = \left[ \begin{matrix} {I & 0 & 0  & 0\cr 0&  0 & 0 & 0 \cr 0 & 0 & I & 0 \cr  0&  0 & 0 & 0 \cr} \end{matrix} \right] \ \ \ {\rm\it and} \ \ \
Q = \left[ \begin{matrix} {0 & 0 & 0 & 0\cr 0 &   I & 0 & 0   \cr  0 & 0 & I & 0 \cr  0&  0 & 0 & 0 \cr} \end{matrix} \right]
$$
with respect to the orthogonal direct sum decomposition $H = H_1 \oplus H_2 \oplus H_3 \oplus H_4$ and assume that $\dim H_1 = \dim H_2 \not=0$, $\dim H_3 \not=0$ and $\dim H_4 \not=0$.

Then there exists
$R \in P_\infty (H)$ such that
$$
\| R - P \| = \| R - Q \| = {1 \over \sqrt{2}}
$$
and there exist more than just one mapping $\gamma\colon [ 0, {\pi \over 2} ] \to P_\infty (H)$ with the properties
$$
\gamma (0) = P, \ \ \ \gamma \left({\pi \over 2} \right) = Q, \ \ \ {\rm\it and} \ \ \  \gamma \left({\pi \over 4} \right) = R,
$$
and
$$
\| \gamma (\theta_1 ) - \gamma (\theta_2) \| = \sin | \theta_1 - \theta_2 |
$$
for all $\theta_1 , \theta_2 \in  [ 0, {\pi \over 2} ]$.
\end{lemma}

\begin{proof}
Since $\dim H_1 = \dim H_2$ we can identify these two subspaces.
Because of $\dim H_3 \not=0$ and $\dim H_4 \not=0$ we can decompose both $H_3$ and $H_4$ into direct sums of a one-dimensional subspace and its orthogonal
complement and then the corresponding matrix representations of $P$ and $Q$ are
$$
P = \left[ \begin{matrix} {I & 0 & 0  & 0 & 0 & 0 \cr 0&  0 & 0 & 0 & 0 & 0 \cr 0 & 0 & 1 & 0 & 0 & 0 \cr  0&  0 & 0 & 0 & 0 & 0 \cr  0&  0 & 0 & 0 & I & 0 \cr  0&  0 & 0 & 0 & 0 & 0 \cr} \end{matrix} \right] \ \ \ {\rm and} \ \ \
Q = \left[ \begin{matrix} {0 & 0 & 0 & 0 & 0 & 0 \cr 0 &   I & 0 & 0 & 0 & 0   \cr  0 & 0 & 1 & 0 & 0 & 0 \cr  0&  0 & 0 & 0 &0 &0\cr  0&  0 & 0 & 0 & I & 0 \cr  0&  0 & 0 & 0 & 0 & 0 \cr} \end{matrix} \right]
$$
with the third and the fourth subspace in the corresponding orthogonal direct sum decomposition of $H$ being one-dimensional.

Choosing
$$
R =  \left[ \begin{matrix} {{1 \over 2}I & {1 \over 2}I & 0  & 0 & 0 & 0 \cr {1 \over 2}I 
&  {1 \over 2}I & 0 & 0 & 0 & 0 \cr 0 & 0 & 1 & 0 & 0 & 0 \cr  0&  0 & 0 & 0 & 0 & 0 \cr  0&  0 & 0 & 0 & I & 0 \cr  0&  0 & 0 & 0 & 0 & 0 \cr} \end{matrix} \right]
$$
we immediately see that for every Lipschitz function
$f\colon [0, {\pi \over 2}] \to [0, {\pi \over 2}]$ with Lipschitz constant $1$ satisfying $f(0) = f({\pi \over 4}) = f({\pi \over 2}) = 0$
the
 mapping $\gamma\colon [ 0, {\pi \over 2} ] \to P_\infty (H)$ given by
$$
\gamma (\theta ) =$$ $$= 
 \left[ \begin{matrix} {(\cos^2 \theta)I & (\cos \theta \sin\theta)I & 0  & 0 & 0 & 0 \cr (\cos\theta \sin\theta) I 
&  (\sin^2 \theta) I & 0 & 0 & 0 & 0 \cr 0 & 0 & \cos^2 f(\theta)  & \cos f(\theta) \sin f(\theta) & 0 & 0 \cr  0&  0 & \cos f(\theta) \sin f(\theta) & \sin^2 f(\theta) & 0 & 0 \cr  0&  0 & 0 & 0 & I & 0 \cr  0&  0 & 0 & 0 & 0 & 0 \cr} \end{matrix} \right]
$$
has the desired properties.
\end{proof}

\begin{lemma}\label{pasedr}
Assume that $P,Q \in P_\infty (H)$ have matrix representations
$$
P = \left[ \begin{matrix} { I & 0 & 0 & 0 & 0 & 0 \cr  0 & 0 & 0 & 0 & 0 & 0 \cr  0 & 0 & I & 0 & 0 & 0 \cr  0 & 0 & 0 & 0 & 0 & 0 \cr  0 & 0 & 0 & 0 & I & 0 \cr  0 & 0 & 0 & 0 & 0 & 0 \cr} \end{matrix} \right]
\ \ \ {\rm\it and} \ \ \ 
Q = \left[ \begin{matrix} { 0 & 0 & 0 & 0 & 0 & 0 \cr  0 & I & 0 & 0 & 0 & 0 \cr  0 & 0 & I & 0 & 0 & 0 \cr  0 & 0 & 0 & 0 & 0 & 0 \cr  0 & 0 & 0 & 0 & C^2 & SC \cr  0 & 0 & 0 & 0 & SC & S^2 \cr} \end{matrix} \right],
$$ 
with respect to the orthogonal direct sum decomposition
 $H= L \oplus L \oplus H_1 \oplus H_2 \oplus K \oplus K$ (note that the first two and the last two summands are equal).
Assume further that $L, K \not=\{0\}$ (while any of $H_1$ and $H_2$ might be a zero subspace and then, of course, some columns and rows in the above matrix representations of $P$ and $Q$
are absent) and that
$S, C \colon K \to K$ are self-adjoint injective operators satisfying $0 \le S,C \le I$ and $S^2 + C^2 = I$.
Suppose finally that $\| S \| = \sin\psi < 1$.

Then there exists
$R \in P_\infty (H)$ such that
$$
\| R - P \| = \| R - Q \| = {1 \over \sqrt{2}}
$$
and there exist more than just one mapping $\gamma \colon [ 0, {\pi \over 2} ] \to P_\infty (H)$ with the properties
$$
\gamma (0) = P, \ \ \ \gamma \left({\pi \over 2} \right) = Q, \ \ \ {\rm\it and} \ \ \  \gamma \left({\pi \over 4} \right) = R,
$$
and
$$
\| \gamma (\theta_1 ) - \gamma (\theta_2) \| = \sin | \theta_1 - \theta_2 |
$$
for all $\theta_1 , \theta_2 \in  [ 0, {\pi \over 2} ]$.
\end{lemma}

\begin{proof}
We set
$$
C_\theta = \cos \left( {2 \theta \over \pi} \arcsin S \right) \ \ {\rm and} \ \  S_\theta = \sin \left( {2 \theta \over \pi} \arcsin S \right), \ \ \theta \in \left[ 0, {\pi \over 2} \right].
$$
Obviously,  there are infinitely many  Lipschitz functions
$f \colon [0, {\pi \over 2}] \to [0, {\pi \over 2}]$ with Lipschitz constant ${\pi \over 2 \psi}$ satisfying $f(0) = 0$, $ f({\pi \over 4}) = {\pi \over 4}$ and $f({\pi \over 2}) = {\pi \over 2}$.
Choose
$$
R =  \left[ \begin{matrix} { {1 \over 2}I &  {1 \over 2}I & 0 & 0 & 0 & 0 \cr   {1 \over 2}I &  {1 \over 2}I & 0 & 0 & 0 & 0 \cr  0 & 0 & I & 0 & 0 & 0 \cr  0 & 0 & 0 & 0 & 0 & 0 \cr  
0 & 0 & 0 & 0 & C_{\pi \over 4}^2 & C_{\pi \over 4} S_{\pi \over 4} \cr  0 & 0 & 0 & 0 & C_{\pi \over 4} S_{\pi \over 4} & S_{\pi \over 4}^2 \cr} \end{matrix} \right]
$$
and the
 mapping $\gamma \colon [ 0, {\pi \over 2} ] \to P_\infty (H)$ given by
$$
\gamma (\theta ) =
\left[ \begin{matrix} { (\cos^2 \theta) I &  (\cos \theta \sin \theta)I & 0 & 0 & 0 & 0 \cr  (\cos \theta \sin \theta)  I &   (\sin^2 \theta)I & 0 & 0 & 0 & 0 \cr  0 & 0 & I & 0 & 0 & 0 \cr  0 & 0 & 0 & 0 & 0 & 0 \cr  
0 & 0 & 0 & 0 & C_{f(\theta)}^2 & C_{f(\theta)} S_{f(\theta)} \cr  0 & 0 & 0 & 0 & C_{f(\theta)} S_{f(\theta)} & S_{f(\theta)}^2 \cr} \end{matrix} \right] .
$$
A straightforward application of Lemma \ref{slomk} completes the proof.
\end{proof}

\begin{lemma}\label{pasetr}
Assume that $P,Q \in P_\infty (H)$ have matrix representations
$$
P = \left[ \begin{matrix} { I & 0 & 0 & 0 & 0 & 0 \cr  0 & 0 & 0 & 0 & 0 & 0 \cr  0 & 0 & I & 0 & 0 & 0 \cr  0 & 0 & 0 & 0 & 0 & 0 \cr  0 & 0 & 0 & 0 & I & 0 \cr  0 & 0 & 0 & 0 & 0 & 0 \cr} \end{matrix} \right]
\ \ \ {\rm\it and} \ \ \ 
Q = \left[ \begin{matrix} { 0 & 0 & 0 & 0 & 0 & 0 \cr  0 & I & 0 & 0 & 0 & 0 \cr  0 & 0 & I & 0 & 0 & 0 \cr  0 & 0 & 0 & 0 & 0 & 0 \cr  0 & 0 & 0 & 0 & C^2 & SC \cr  0 & 0 & 0 & 0 & SC & S^2 \cr} \end{matrix} \right],
$$ 
with respect to the orthogonal direct sum decomposition
 $H= L \oplus L \oplus H_1 \oplus H_2 \oplus K \oplus K$.
Assume further that $K \not=\{0\}$ and that
$S, C \colon K \to K$ are self-adjoint injective operators satisfying $0 \le S,C \le I$ and $S^2 + C^2 = I$.
Suppose finally that $\| S \| =  1$.

Then there exists
$R \in P_\infty (H)$ such that
$$
\| R - P \| = \| R - Q \| = {1 \over \sqrt{2}}
$$
and there exist more than just one mapping $\gamma \colon [ 0, {\pi \over 2} ] \to P_\infty (H)$ with the properties
$$
\gamma (0) = P, \ \ \ \gamma \left({\pi \over 2} \right) = Q, \ \ \ {\rm\it and} \ \ \  \gamma \left({\pi \over 4} \right) = R,
$$
and
$$
\| \gamma (\theta_1 ) - \gamma (\theta_2) \| = \sin | \theta_1 - \theta_2 |
$$
for all $\theta_1 , \theta_2 \in  [ 0, {\pi \over 2} ]$.
\end{lemma}

\begin{proof}
Note that $\| S \| = 1$, $S^2 + C^2 =I$, and the injectivity of both $S$ and $C$ impy that $K$ is infinite-dimensional. By the spectral theorem the bottom-right corners of $P$ and $Q$:
$$
\left[ \begin{matrix} { I & 0 \cr 0 & 0 \cr} \end{matrix} \right] \ \ \ {\rm and} \ \ \ \left[ \begin{matrix} { C^2 & CS \cr CS & S^2 \cr} \end{matrix} \right] 
$$
can be rewritten as
$$
\left[ \begin{matrix} { I & 0 & 0 & 0\cr 0 & 0 & 0 & 0 \cr  0 & 0 & I & 0\cr 0 & 0 & 0 & 0 \cr} \end{matrix} \right] \ \ \ {\rm and} \ \ \ \
\left[ \begin{matrix} { C_{1}^2 & C_1 S_1 & 0 & 0 \cr C_1 S_1 & S_{1} ^2 & 0 & 0 \cr 0 & 0 &  C_{2}^2 & C_2 S_2  \cr 0 & 0& C_2 S_2 & S_{2} ^2  \cr} \end{matrix} \right] 
$$
with $\| S_1 \| = 1$ and $0 < \| S_2 \| < 1$.
One can now use the same ideas as before to complete the proof.
\end{proof}

For $P,Q \in P_\infty (H)$ we write $P \sim Q$ if and only if
$P\perp Q$ or $ (I-P) \perp (I-Q)$, and $P \sharp Q$ if and only if $P \perp Q$ and $P +Q \in P_\infty (H)$,
that is, we have $P \sharp Q$ if and only if $P$ and $Q$ are orthogonal and the kernel of $P+Q$ is infinite-dimensional, or in other words, $P+Q$ have up to a unitary similarity matrix
representations
$$
P = \left[ \begin{matrix} {I & 0 & 0 \cr 0 & 0 & 0 \cr 0 & 0 & 0 \cr} \end{matrix} \right] \ \ \ {\rm and} \ \ \
Q = \left[ \begin{matrix} {0 & 0 & 0 \cr 0 & I & 0 \cr 0 & 0 & 0 \cr} \end{matrix} \right],
$$  
where the three direct summands in the corresponding orthogonal direct sum decomposition of $H$ are all infinite-dimensional.
Note that above we could define $P \sharp Q$  if and only if $P +Q \in P_\infty (H)$, since
$P+Q \in P_\infty(H)$ yields $P\perp Q$.

\begin{lemma}\label{minneap}
Let $P,Q \in P_\infty (H)$. Then there exist $j \in \{ 1,2,3\}$ and a  sequence $P=P_0, P_1, \ldots, P_j =Q$ such
that $P_k \in P_\infty (H)$, $k=0,1,\ldots , j$, and $P_{k-1} \perp P_k$, $k=1, \ldots, j$.
\end{lemma}

\begin{proof}
By the two projections theorem we have up to a unitary similarity matrix representations
$$
P = \left[ \begin{matrix} { I & 0 & 0 & 0 & 0 & 0 \cr  0 & 0 & 0 & 0 & 0 & 0 \cr  0 & 0 & I & 0 & 0 & 0 \cr  0 & 0 & 0 & 0 & 0 & 0 \cr  0 & 0 & 0 & 0 & I & 0 \cr  0 & 0 & 0 & 0 & 0 & 0 \cr} \end{matrix} \right]
\ \ \ {\rm and} \ \ \ 
Q = \left[ \begin{matrix} { 0 & 0 & 0 & 0 & 0 & 0 \cr  0 & I & 0 & 0 & 0 & 0 \cr  0 & 0 & I & 0 & 0 & 0 \cr  0 & 0 & 0 & 0 & 0 & 0 \cr  0 & 0 & 0 & 0 & C^2 & SC \cr  0 & 0 & 0 & 0 & SC & S^2 \cr} \end{matrix} \right],
$$ 
where the underlying orthogonal decomposition of $H$ is $H= H_1 \oplus H_2 \oplus H_3 \oplus H_4 \oplus K \oplus K$ and
$S, C \colon K \to K$ are self-adjoint injective operators satisfying $0 \le S,C \le I$ and $S^2 + C^2 = I$.

In the case when $K$ is infinite-dimensional
the bottom-right corners of $P$ and $Q$:
$$
\left[ \begin{matrix} { I & 0 \cr 0 & 0 \cr} \end{matrix} \right] \ \ \ {\rm and} \ \ \ \left[ \begin{matrix} { C^2 & CS \cr CS & S^2 \cr} \end{matrix} \right] 
\in B(K\oplus K)
$$
can be rewritten as
$$
\left[ \begin{matrix} { I & 0 & 0 & 0\cr 0 & 0 & 0 & 0 \cr  0 & 0 & I & 0\cr 0 & 0 & 0 & 0 \cr} \end{matrix} \right] \ \ \ {\rm and} \ \ \ \
\left[ \begin{matrix} { C_{1}^2 & C_1 S_1 & 0 & 0 \cr C_1 S_1 & S_{1} ^2 & 0 & 0 \cr 0 & 0 &  C_{2}^2 & C_2 S_2  \cr 0 & 0& C_2 S_2 & S_{2} ^2  \cr} \end{matrix} \right]
\in B(K\oplus K)
$$
where all the underlying direct summands are infinite-dimensional,
and then the projections $P_0=P$, $P_3 = Q$, $P_1$ whose matrix representation has bottom-right corner equal to
$$
\left[ \begin{matrix} { 0 & 0 & 0 & 0\cr 0 & I & 0 & 0 \cr  0 & 0 & 0 & 0\cr 0 & 0 & 0 & 0 \cr} \end{matrix} \right]
\in B(K\oplus K)
$$
and zeros elsewhere, and $P_2$ whose bottom-right corner equals
$$
\left[ \begin{matrix} { 0 & 0 & 0 & 0 \cr 0 & 0 & 0 & 0 \cr 0 & 0 &  S_{2}^2 & -C_2 S_2  \cr 0 & 0& -C_2 S_2 & C_{2} ^2  \cr} \end{matrix} \right] 
\in B(K\oplus K)
$$
while all other entries are zero, do the job.

In the case when $H_4$ is infinite-dimensional we simply choose $P_0=P$, $P_2 = Q$, and $P_1$ is a projection whose image is $H_4$ to complete the proof.

It remains to consider the case when both $K$ and $H_4$ are finite-dimensional. Then both $H_1$ and $H_2$ must be infinite-dimensional. This case is easy and is left to the reader.
\end{proof}

\begin{corollary}\label{zamuda}
Let $P,Q \in P_\infty (H)$. Then there exist $j \in \{ 1,2,3\}$ and a sequence $P=P_0, P_1, \ldots, P_j =Q$ such
that $P_k \in P_\infty (H)$, $k=0,1,\ldots , j$, and $P_{k-1} \sharp P_k$, $k=1, \ldots, j$.
\end{corollary}

\begin{proof}
By the previous lemma we can find a sequence  $P=Q_0, Q_1, \ldots, Q_j =Q$ such
that $Q_k \in P_\infty (H)$, $k=0,1,\ldots , j$, and $Q_{k-1} \perp Q_k$, $k=1, \ldots, j$.
There is no loss of generality in assuming that
$j > 1$. Indeed, in the case when $j=1$ we have $P \perp Q$
and we can choose the sequence $Q_0 = P$, $Q_1 = Q$,
$Q_2 = P$, and $Q_3 = Q$. Thus, we may assume that $j \in \{ 2,3 \}$
and all we need to do is to replace $Q_1, \ldots, Q_{j-1}$ by projections 
 $P_1, \ldots, P_{j-1}$ satisfying $P_k \le Q_k$, $k=1, \ldots, j-1$, and the image of $Q_k - P_k$ is
infinite-dimensional.
\end{proof}

In the next few lemmas we will always assume that $\phi \colon P_\infty (H) \to P_\infty (H)$ is a bijective map preserving the relation $\sim$ in both directions, that is,
$$
P \sim Q \iff \phi (P) \sim \phi (Q) , \ \ \ P,Q \in P_\infty (H).
$$

\begin{lemma}\label{organi}
Let $P_j \in P_\infty (H)$, $j=1,2,3$, and assume that $P_k \sharp P_m$ whenever $k\not=m$. Then either
 $\phi (P_k) \sharp \phi (P_m )$ whenever $k\not=m$, or 
 $(I - \phi (P_k))  \sharp (I- \phi (P_m ))$ whenever $k\not=m$.
\end{lemma}

\begin{proof}
Of course, we only need to verify that
 either
 $\phi (P_k) \perp \phi (P_m )$ whenever $k\not=m$, or 
 $(I - \phi (P_k))  \perp (I- \phi (P_m ))$ whenever $k\not=m$.

For each pair $k,m$, $k\not=m$, we have
 $\phi (P_k) \perp \phi (P_m )$   or 
 $(I - \phi (P_k))  \perp (I- \phi (P_m ))$. It follows that there exists
$r \in \{ 1,2,3 \}$ such that either
$\phi (P_r) \perp \phi (P_m )$ for both integers $m \not=r$, or
$(I- \phi (P_r)) \perp (I- \phi (P_m ))$ for both integers $m \not=r$.

We will consider only the first possibility and we will assume with no loss of generality that $r=1$.
Hence, we have
$$
\phi (P_1) \perp \phi (P_2) \ \ \ {\rm and} \ \ \ \phi (P_1 ) \perp \phi (P_3).
$$
Clearly, $I - \phi (P_2)$ and $I - \phi (P_3)$ cannot be orthogonal, as their images contain ${\rm Im}\, \phi(P_1)$, forcing $\phi (P_2)$ and $\phi (P_3)$ to be orthogonal.
\end{proof}

\begin{lemma}\label{lukna}
Let $P,Q,R \in P_\infty (H)$ satisfy $P \sharp Q$, $P \sharp R$, and $\phi (P) \sharp \phi (Q)$. Then 
$$
\phi (P) \sharp \phi (R).
$$
\end{lemma}

\begin{proof}
With respect to the orthogonal direct sum decomposition $H = {\rm Im}\, P \oplus {\rm Ker}\, P$ we have
$$
P = \left[ \begin{matrix} { I & 0 \cr 0 & 0 \cr} \end{matrix} \right], \ \ \ 
Q = \left[ \begin{matrix} { 0 & 0 \cr 0 & Q_1 \cr} \end{matrix} \right], \ \ \ {\rm and} \ \ \ 
R = \left[ \begin{matrix} { 0 & 0 \cr 0 & R_1 \cr} \end{matrix} \right] 
$$
with $Q_1 , R_1 \in P_\infty ({\rm Ker}\, P)$. According to Lemma \ref{minneap} we can find
$Q_1=T_0, T_1, \ldots, T_j =R_1$ such
that $T_k \in P_\infty ({\rm Ker}\, P)$, $k=0,1,\ldots , j$, and $T_{k-1} \perp T_k$, $k=1, \ldots, j$.
Set
$$
S_k =  \left[ \begin{matrix} { 0 & 0 \cr 0 & T_k \cr} \end{matrix} \right], \ \ \  k=0, 1, \ldots , j.
$$
Then $S_0 = Q$, $S_j = R$, and $S_{k-1} \sharp S_k$, $k=1, \ldots , j$, and $P \sharp S_k$, $k=0, \ldots , j$.

Using Lemma \ref{organi} for the triple $P, Q= S_0, S_1$ we see that $\phi (P) \sharp \phi (S_1)$. If $j=1$ we are done. Otherwise
we apply Lemma
\ref{organi} once more, this time for the triple $P, S_1, S_2$ to conclude that $\phi (P) \sharp \phi (S_2)$.
If $j=2$ we are done. Otherwise we need one more step to conclude the proof.
\end{proof}

For $P\in P_\infty (H)$ we denote $$
P^\sharp = \{ Q \in P_\infty (H) \colon \, P \sharp Q \}.
$$

\begin{lemma}\label{fupavre}
Assume that there exist $P, Q \in P_\infty (H)$ such that $P \sharp Q$ and $\phi (P) \sharp \phi (Q)$.
Then for every $R \in P_\infty (H)$ we have
$$
\phi (R^\sharp) = (\phi (R))^\sharp,
$$
that is, $\phi$ preserves the relation $\sharp$ in both directions.
\end{lemma}

\begin{proof}
If for some $R\in P_\infty (H)$ there exists $S\in R^\sharp$ such that $\phi (R) \sharp \phi (S)$, then Lemma \ref{lukna} 
yields that $\phi (R^\sharp) \subset (\phi (R))^\sharp$. But as $\phi^{-1}$ has the same properties as $\phi$, we actually have
$\phi (R^\sharp) = (\phi (R))^\sharp$.

It follows that if for some $R\in P_\infty (H)$ we have $\phi (R^\sharp) = (\phi (R))^\sharp$, then 
 $\phi (S^\sharp) = (\phi (S))^\sharp$ for every $S\in R^\sharp$.

The lemma is now a straightforward consequence of Corollary \ref{zamuda}.
\end{proof}

\begin{corollary}\label{silv}
Let $\phi \colon P_\infty (H) \to P_\infty (H)$ be a bijective map such that for every pair $P,Q \in P_\infty (H)$ we have
$$
P \sim Q \iff \phi (P) \sim \phi (Q).
$$
 Then there exists a unitary or an antiunitary operator (orthogonal operator) $U$ on $H$ such that either 
$$
\phi (P) = UPU^\ast
$$
for every $P \in P_\infty (H)$; or
$$
\phi (P) = U (I-P) U^\ast
$$
for every $P \in P_\infty (H)$.
\end{corollary}

\begin{proof}
We take $P,Q \in P_\infty (H)$ with $P \sharp Q$. By Lemma \ref{organi} we have either $\phi (P) \sharp \phi (Q)$,
or $(I- \phi (P)) \sharp (I - \phi (Q))$. In the second case we replace the map $\phi$ by the map $P \mapsto I - \phi (P)$.
Thus, there is no loss of generality in assuming that we have the first possibility. But then we already know that
for every $S,T \in P_\infty (H)$ we have
\begin{equation}\label{jjhhl}
T \sharp S \iff \phi (T) \sharp \phi (S).
\end{equation}

We will prove  that
for every $S,T \in P_\infty (H)$ we have
$$
T \perp S \iff \phi (T) \perp \phi (S).
$$
Once we will verify this the conclusion follows from
\cite[Theorem 1.2]{Sem} (or from a result of \cite{Gy}).

Because the inverse of $\phi$ has the same properties as $\phi$, it is enough to check that for every $S,T \in P_\infty (H)$ we have
$$
T \perp S \Rightarrow \phi (T) \perp \phi (S).
$$

Let $T,S$ be any elements of $P_\infty(H)$. Clearly,
$S \le T$ if and only if for every $M \in P_\infty (H)$ we have
$$
T \sharp M \Rightarrow S \sharp M,
$$
and hence (\ref{jjhhl}) yields that
$$
S \le T \iff \phi (S) \le \phi (T).
$$

Let $T \in P_\infty (H)$. Then $I-T$ can be characterized as the unique element $L$ of $P_\infty (H)$ with the properties:
\begin{itemize}
\item for every $M \in P_\infty (H)$ we have: $M \sharp T \Rightarrow M \le L$, and
\item for every pair $M,N \in P_\infty (H)$ we have: $M,N \le T$ and $M \sharp N$ imply that $M \sharp L$ and $N \sharp L$.
\end{itemize}

It follows that $\phi (I-T) = I - \phi (T)$. Since $T \perp S$ holds if and only $S \le I -T$ we conclude that $\phi$ preserves orthogonality.
\end{proof}

\section{Proofs of the main results}

We start with the proof of our main result.

\begin{proof}[Proof of Theorem \ref{glavni}]
Assume that $P,Q \in P_\infty (H)$ and $P\sim Q$. We already know from the proof of Lemma \ref{and} that then there exists an $R\in P_\infty (H)$ with
\begin{equation}\label{holds}
\| R - P \| = \| R - Q \| = {1 \over \sqrt{2}}.
\end{equation}
Note that for such an $R$ we also have
\begin{equation}\label{kepek}
\| \phi(R) - \phi(P) \| = \| \phi(R) - \phi(Q) \| = {1 \over \sqrt{2}}.
\end{equation}

By Halmos' two projections theorem, up to a unitary similarity $H$ can be written as an orthogonal direct sum $H= H_1 \oplus H_2 \oplus H_3 \oplus H_4 \oplus K \oplus K$ such that the corresponding matrix representations
of $\phi(P)$ and $\phi(Q)$ are
$$
\phi( P) = \left[ \begin{matrix} { I & 0 & 0 & 0 & 0 & 0 \cr  0 & 0 & 0 & 0 & 0 & 0 \cr  0 & 0 & I & 0 & 0 & 0 \cr  0 & 0 & 0 & 0 & 0 & 0 \cr  0 & 0 & 0 & 0 & I & 0 \cr  0 & 0 & 0 & 0 & 0 & 0 \cr} \end{matrix} \right]
\ \ \ {\rm and} \ \ \ 
\phi (Q) = \left[ \begin{matrix} { 0 & 0 & 0 & 0 & 0 & 0 \cr  0 & I & 0 & 0 & 0 & 0 \cr  0 & 0 & I & 0 & 0 & 0 \cr  0 & 0 & 0 & 0 & 0 & 0 \cr  0 & 0 & 0 & 0 & C^2 & SC \cr  0 & 0 & 0 & 0 & SC & S^2 \cr} \end{matrix} \right],
$$ 
where $S, C \colon K \to K$ are self-adjoint injective operators satisfying $0 \le S,C \le I$ and $S^2 + C^2 = I$.

By (\ref{kepek}) and Lemma \ref{etko} we have necessarily $\dim H_1 = \dim H_2$. Furthermore, from $\| \phi (P) - \phi (Q) \| = 1$ and Proposition \ref{ninana-es-jiba} we immediately conclude that either $\dim H_1 = \dim H_2\not= 0$,
or $K\not= \{ 0 \}$ and $\| S\| = 1$.

Next, we prove that the second possibility cannot happen.
Let $\widetilde R\in P_\infty (H)$ be arbitrary such that it satisfies
$$
\| \widetilde R - \phi(P) \| = \| \widetilde R - \phi(Q) \| = {1 \over \sqrt{2}}.
$$
Note that by (\ref{kepek}) such an $\widetilde{R}$ exists.
We set $R:= \phi^{-1}(\widetilde{R})$ which clearly satisfies (\ref{holds}).
By Lemma \ref{and} there exists exactly one mapping $\alpha \colon [ 0, {\pi \over 2} ] \to P_\infty (H)$ such that
$\alpha (0) = P$, $\alpha \left({\pi \over 2} \right) = Q$,  $\alpha \left({\pi \over 4} \right) = R$,
and $\| \alpha (\theta_1 ) - \alpha (\theta_2) \| = \sin | \theta_1 - \theta_2 |$
for all $\theta_1 , \theta_2 \in  [ 0, {\pi \over 2} ]$. 
Since $\phi$ is a bijective isometry, it is straightforward that there exists exactly one mapping $\gamma \colon [ 0, {\pi \over 2} ] \to P_\infty (H)$ such that
$\gamma (0) = \phi(P)$, $\gamma \left({\pi \over 2} \right) = \phi(Q)$,  $\gamma \left({\pi \over 4} \right) = \widetilde{R}$,
and $\| \gamma (\theta_1 ) - \gamma (\theta_2) \| = \sin | \theta_1 - \theta_2 |$
for all $\theta_1 , \theta_2 \in  [ 0, {\pi \over 2} ]$.
Using Lemma \ref{pasetr} we see that we cannot have $K\not= \{ 0 \}$ and $\| S\| = 1$.

Hence, we have $\dim H_1 = \dim H_2\not= 0$, furthermore if $K\not= \{ 0 \}$, then $\|S\| < 1$.
In a similar way as above, we get from Lemma \ref{pasedr} that $K$ actually must be the zero subspace.
Now, using the same argument once more, we conclude from Lemma \ref{paseen} that $\dim H_3 = 0$ or $\dim H_4 = 0$.
Thus up to a unitary similarity we have either
$$
\phi(P) = \left[ \begin{matrix} {I & 0 & 0 \cr 0 & 0 & 0 \cr 0 & 0 & 0 \cr} \end{matrix} \right] \ \ \ {\rm and} \ \ \
\phi(Q) = \left[ \begin{matrix} {0 & 0 & 0 \cr 0 & I & 0 \cr 0 & 0 & 0 \cr} \end{matrix} \right]
$$  
(note that the last row and the last column may be absent), or
$$
\phi (P) = \left[ \begin{matrix} {I & 0 & 0 \cr 0 & 0 & 0 \cr 0 & 0 & I \cr} \end{matrix} \right] \ \ \ {\rm and} \ \ \
\phi (Q) = \left[ \begin{matrix} {0 & 0 & 0 \cr 0 & I & 0 \cr 0 & 0 & I \cr} \end{matrix} \right] .
$$
In other words, $P\sim Q$ implies $\phi (P)\sim\phi (Q)$.
Since $\phi^{-1}$ has the same properties as $\phi$, we obtain that
$$
P \sim Q \iff \phi (P) \sim \phi (Q)
$$
for every pair $P,Q \in P_\infty (H)$.
We complete the proof using Corollary \ref{silv}.
\end{proof}

\begin{proof}[Proof of Theorem \ref{vse}]
We need to prove that 
\begin{itemize}
\item either $0$ is mapped to $0$ and $I$ is mapped to $I$, or  $0$ is mapped to $I$ and $I$ is mapped to $0$,
\item for every positive integer $n$ either $P_n (H)$ is mapped onto $P_n (H)$ and $P^n (H)$ is mapped onto $P^n (H)$, or
$P_n (H)$ is mapped onto $P^n (H)$ and $P^n (H)$ is mapped onto $P_n (H)$,
\item and $\phi (P_\infty (H)) = P_\infty (H)$.
\end{itemize}
Once we will verify that the above is true the proof follows directly from structural results for bijective isometries of $P_n (H)$, $n=1,2, \ldots$, 
Theorem \ref{glavni}, and the obvious fact that $\varphi \colon P_n (H) \to P_n (H)$ is a bijective isometry if and only if the maps $P \mapsto I-\varphi (P)$, $P\in P_n (H)$,
and $P \mapsto I - \varphi (I -P)$, $P \in P^n (H)$, are bijective isometries of $P_n (H)$ onto $P^n (H)$, and $P^n (H)$ onto itself, respectively.

If $P,Q \in P(H)$ belong to two different subsets 
\begin{equation}\label{ekskurzn}
\{ 0 \}, \{ I \}, P_1 (H), P^1 (H), 
 P_2 (H), P^2 (H), \ldots, P_\infty (H), 
\end{equation}
then clearly $\| P - Q \| =1$. On the other hand, 
if $P$ and $Q$ belong to the same subset, then one can easily find a chain of projections $P = P_0, P_1, \ldots, P_k = Q$ with $k$ a positive integer such that
$\| P_{j-1} - P_j \| < 1$, $j=1, \ldots , k$. Of course, this is trivial when $\| P -Q \| < 1$ and it follows easily from the two projections theorem when
$\| P  - Q \| = 1$.

Hence, each of the subsets from the list (\ref{ekskurzn}) is mapped bijectively onto some subset in the list. In this list we have two singletons and then of course,
each of them is mapped onto itself, or each of them is mapped onto the other one. It was proved in \cite{GeS} that if $\| P - Q \| = 1$ for some positive integer $n$ and some $P,Q \in P_n (H)$,
then $P^{\le { 1 \over \sqrt{2}}}\cap Q^{\le { 1 \over \sqrt{2}}}$ is either a compact manifold homeomorphic to the unitary group ${\cal U}_n$ of all $n \times n$ unitary matrices, or
a non-empty non-compact subset of $P_n (H)$. As $P_n (H)$ is isometric to $P^n (H)$ we immediately conclude that for every pair $P, Q \in P^n (H)$ with $\| P - Q \| = 1$ the set
 $P^{\le { 1 \over \sqrt{2}}}\cap Q^{\le { 1 \over \sqrt{2}}}$ is either a compact manifold with the same dimension as ${\cal U}_n$, or
a non-empty non-compact subset of $P_n (H)$. The dimensions of the compact manifolds ${\cal U}_n$, $n=1,2, \ldots$, are well-known. For us the exact values are not important, we only
need the obvious fact that the dimension of ${\cal U}_n$ is a strictly increasing function of $n$. Finally, Lemma \ref{etko} tells that for every $P \in P_\infty (H)$ we can find  $Q \in P_\infty (H)$
such that $\| P - Q\| = 1$ and 
 $P^{\le { 1 \over \sqrt{2}}}\cap Q^{\le { 1 \over \sqrt{2}}}$ is the empty set.

It follows that for every positive integer $n$ either $P_n (H)$ is mapped onto $P_n (H)$ and $P^n (H)$ is mapped onto $P^n (H)$, or
$P_n (H)$ is mapped onto $P^n (H)$ and $P^n (H)$ is mapped onto $P_n (H)$,
and $\phi (P_\infty (H)) = P_\infty (H)$, as desired.
\end{proof}

\bigskip

\noindent{\bf Acknowledgement.} The authors are very grateful to the London Mathematical Society for the Research in Pairs (Scheme 4) grant (reference number: 41642) which enabled the second author to visit the first author at University of Reading. During that visit the authors were able to complete this research project on which they have been working over the last few years.


\begin{thebibliography}{99}



\bibitem{BJM} 
F. Botelho, J. Jamison, and L. Moln\' ar, 
Surjective isometries on Grassmann spaces,
{\em J. Funct. Anal.} {\bf 265} (2013), 2226-2238. 

\bibitem{BS} 
A. B\"ottcher and I.M. Spitkovsky,
A gentle guide to the basics of two projections theory,
{\em Linear Algebra Appl.} {\bf 432} (2010), 1412--1459.

\bibitem{Geh}
G.P. Geh\' er, 
Wigner's theorem on Grassmann spaces,
{\em J. Funct. Anal.} {\bf 273} (2017), 2994--3001.

\bibitem{GeS}
 G.P. Geh\' er and P. \v Semrl, 
Isometries of Grassmann spaces,  
{\em J. Funct. Anal.} {\bf 270}  (2016), 1585-1601.

\bibitem{Gy}
M. Gy\H{o}ry, 
Transformations on the set of all $n$-dimensional subspaces of a Hilbert space preserving orthogonality,
\emph{Publ. Math. Debrecen} \textbf{65} (2004), 233--242.

\bibitem{Mo}
L. Moln\'ar,
Transformations on the set of all $n$-dimensional subspaces of a Hilbert space preserving principal angles,
{\em Comm. Math. Phys.} {\bf 217} (2001), 409--421.

\bibitem{Sem} 
P. \v Semrl, 
Orthogonality preserving transformations on the set of $n$-dimensional subspaces of a Hilbert space,
{\em Illinois J. Math.} {\bf 48} (2004), 567--573.


\end{thebibliography}
\end{document}